\documentclass[a4paper,12pt]{amsart}

\usepackage[margin=1in]{geometry}

\usepackage{amsmath, amsthm, amssymb}
\usepackage{cite}
\usepackage[dvipdfmx]{graphicx}
\usepackage{bm}

\theoremstyle{definition}
\newtheorem{theorem}{Theorem}[section]
\newtheorem{lemma}[theorem]{Lemma}

\newtheorem{fact}[theorem]{Fact} 
\newtheorem{observation}[theorem]{Observation} 

\theoremstyle{definition}
\newtheorem{definition}[theorem]{Definition}

\newtheorem{pclaim}[theorem]{Claim}
\newtheorem*{ac}{Acknowledgments} 

\theoremstyle{remark}

\newenvironment{rmenum}{
\begin{enumerate}

}
{\end{enumerate}}

\newcommand{\parNei}[2]{N_{#1}(#2)}

\newcommand{\tcomp}[2]{\mathcal{G}(#1, #2)}

\newcommand{\distf}[4]{\lambda(#3, #4; #2; #1)}
\newcommand{\distgt}[4]{\lambda(#3, #4; #1, #2)}

\newcommand{\distgtf}[5]{\lambda(#4, #5; #3; #1, #2)}
\newcommand{\parcut}[2]{\delta_{#1}(#2)}

\newcommand{\gtsim}[2]{\sim_{(#1, #2)}}
\newcommand{\tpart}[2]{\mathcal{P}(#1, #2)}
\newcommand{\gtpart}[2]{\mathcal{P}(#1, #2)}

\newcommand{\conn}[1]{\mathcal{C}(#1)}
\newcommand{\conng}[2]{\mathcal{C}_{#1}(#2)}
\newcommand{\connp}[1]{\mathcal{C}^{*}(#1)}

\newcommand{\laygtr}[4]{U_{ < #4}(#3; #1, #2)}
\newcommand{\layler}[2]{U_{ \le #1}(#2)}
\newcommand{\laylegtr}[4]{U_{ \le #4}(#3; #1, #2)}
\newcommand{\levelr}[2]{U_{#1}(#2)}
\newcommand{\levelgtr}[4]{U_{#4}(#3; #1, #2)}

\newcommand{\laycompgtr}[4]{\mathcal{L}_{#4}( #3; #1, #2)}
\newcommand{\laycompall}[3]{\mathcal{L}_{(#1, #2)}(#3)}

\newcommand{\agtr}[3]{A_{(#1, #2)}(#3)}
\newcommand{\apr}[2]{A_{#1}(#2)}
\newcommand{\dgtr}[3]{D_{(#1, #2)}(#3)}

\newcommand{\cgtr}[3]{C_{(#1, #2)}(#3)}

\newcommand{\extend}[5]{(#1, #2; #3) \otimes (#4, #5)}  
 
\newcommand{\initialgtr}[3]{K_{(#1, #2)}(#3)}

\newcommand{\coupgt}[3]{D^\circ_{(#1, #2)}(#3)} 
\newcommand{\coupgtf}[4]{D^\circ_{(#1, #2)}(#4; #3)} 
  
\newcommand{\neicompgt}[3]{\mathcal{B}_{(#1, #2)}(#3)}  
\newcommand{\neicompgtr}[4]{\mathcal{B}_{(#1, #2)}(#4; #3)}  
  
\newcommand{\neisetgt}[3]{B_{(#1, #2)}(#3)}  
\newcommand{\neisetgtr}[4]{B_{(#1, #2)}(#4;  #3)}

\newcommand{\setdistgtf}[5]{\lambda(#4, #5; #3; #1, #2)}

\title[Tight Cuts in Bipartite Grafts I]{Tight Cuts in Bipartite Grafts I: Capital Distance Components}
\author{Nanao Kita}
\address{Tokyo University of Science 2641 Yamazaki, Noda, Chiba, Japan 278-0022}
\email{kita@rs.tus.ac.jp}

\date{\today}

\begin{document}

\begin{abstract} 
This paper is the first from a series of papers that provide a characterization of maximum packings of $T$-cuts in bipartite graphs.    
Given a connected graph, a set $T$ of an even number of vertices, and a minimum $T$-join, an edge weighting can be defined, 
   from which distances between vertices can be defined.   
Furthermore, given a specified vertex called root, 
vertices can be classified  according to their distances from the root,  
and this classification of vertices can be used to define a family of subgraphs called {\em distance components}.  
Seb\"o provided a theorem that revealed a relationship between distance components, minimum $T$-joins, and $T$-cuts.  
In this paper, 
we further investigate the structure of distance components in bipartite graphs. 
Particularly, we focus on {\em capital} distance components, that is, those that include the root. 
We reveal the structure of capital distance components  in terms of 
the $T$-join analogue of the general Kotzig-Lov\'asz canonical decomposition. 
\end{abstract}

\maketitle

\section{Definitions}

Let $G$ be a graph. 
The vertex and edge sets of $G$ are denoted by $V(G)$ and $E(G)$, respectively. 
We consider multigraphs. 
For two distinct vertices $x, y \in V(G)$,  an edge between $x$ and $y$ can be denoted by $xy$. 
We denote the set of connected components of $G$ by $\conn{G}$. 
The set of connected components of $G$ with at least one edge is denoted by $\connp{G}$. 
We treat path and circuits as graphs. 
Given a path $P$ and two vertices $x, y\in V(P)$, 
$xPy$ denotes the subpath of $P$ between $x$ and $y$. 
We sometimes treat graphs as the set of its vertices.

Let $X \subseteq V(G)$. 
The subgraph of $G$ induced by $X$ is denoted as $G[X]$. 
The graph $G[V(G)\setminus X]$ can be denoted by $G -X$. 
The set $\conn{G[X]}$ can be denoted as $\conng{G}{X}$. 
A path $P$ with two distinct ends $x$ and $y$ is called a {\em round ear path relative to} $X$ if $V(P)\cap X = \{x, y\}$. 
These $x, y$ are called {\em bonds} of $P$.

Let $\hat{G}$ be a supergraph of $G$, and let $F\subseteq E(\hat{G})$. 
The sum $G + F$ denotes the graph obtained by adding $F$ to $G$.  
In contrast, if $F\subseteq E(G)$ holds, then $G. F$ denotes the subgraph of $G$ determined by $F$. 
Let $G_1$ and $G_2$ be two subgraphs of $\hat{G}$. 
We denote the sum of  $G_1$ and $G_2$ by $G_1 + G_2$.

A neighbor of $X$ is a vertex from $V(G)\setminus X$ that is adjacent to a vertex from $X$. 
The set of neighbors of $X$ is denoted by $\parNei{G}{X}$. 
The set of edges between $X$ and $V(G)\setminus X$ is denoted by $\parcut{G}{X}$. 
The set of ends of edges from $F$ is denoted by $\partial_G(F)$.

The set of integers is denoted by $\mathbb{Z}$. 
The symmetric difference of two sets $A$ and $B$ is denoted by $A\Delta B$. 
That is, $A \Delta B$ denotes $(A\setminus B) \cup (B \setminus A)$. 
We often denote a singleton $\{x\}$ simply by $x$.

\section{Grafts and Joins}

Let $G$ be a graph, and let $T\subseteq V(G)$. 
A set $F\subseteq E(G)$ is a {\em join} of the pair $(G, T)$ 
if $|\parcut{G}{v} \cap F|$ is odd for every $v\in T$ 
and even for every $v\in V(G)\setminus T$. 
A join is said to be {\em minimum}  if it has the minimum number of edges. 
Pair $(G, T)$ is called a {\em graft} if  $|K\cap V(T)|$ is even for every $K \in \conn{G}$.  
A pair $(G, T)$ of  graph $G$ and set $T\subseteq V(G)$ has a join 
if and only if it is a graft.  
For a graft $(G, T)$, 
the number of edges in a minimum join of $(G, T)$ is denoted by $\nu(G, T)$. 
We often treat items or properties of $G$ as those of $(G, T)$. 
A graft $(G, T)$ is said to be connected if $G$ is connected. 
A graft $(G, T)$ is {\em bipartite} 
if $G$ is bipartite. 
Color classes of $G$ are referred to as color classes of $(G, T)$.

Let $(G, T)$ be a graft. 
Let $H$ be a subgraph of $G$, and let $S \subseteq T$. 
The pair $(H, S)$ is called a subgraft of $(G, T)$ if $(H, S)$ is a graft.

Let $F\subseteq E(G)$ be a join of $(G, T)$, and let $X \subseteq V(G)$.  
Define $Y \subseteq X$ as follows: 
A vertex $v \in X$ is a member of $Y$  if $|T \cap \{v\}|$ and 
the number of edges from $F$ between $v$ and $V(G)\setminus X$ are of distinct parities. 
Then, the pair $(G[X], Y)$ is a subgraft of $(G, T)$ 
and is denoted by $(G, T)_F[X]$.

An edge is said to be {\em allowed} if  $(G, T)$ has a minimum join that contains this edge. 
A subgraph $H$ of $G$ is {\em factor-connected} if, 
for every two vertices $x, y\in V(H)$, $H$ has a path between $x$ and $y$ in which every edge is allowed. 
A maximal factor-connected subgraph of $G$ is called a {\em factor-component} of $(G, T)$. 
The set of factor-components of $(G, T)$ is denoted by $\tcomp{G}{T}$.  
Graft $(G, T)$ is said to be {\em factor-connected} if $G$ is factor-connected.

\section{Weights and Distances over Grafts}

Hereinafter in this paper, we assume that every graft is connected. 

\begin{definition} 
Let $(G, T)$ be a graft. Let $F \subseteq E(G)$. 
We define $w_F: E(G)\rightarrow \{1, -1\}$ as  
$w_F(e) = 1$ for $e\in E(G)\setminus F$ and $w_F(e) = -1$ for $e\in F$. 
For a subgraph $P$ of $G$, which is typically a path or circuit, 
 $w_F(P)$ denotes $ \Sigma_{ e \in E(P) } w_F(e)$, 
 and is referred to as the $F$-{\em weight} of $P$. 

For $u,v\in V(G)$, 
a path between $u$ and $v$ with the minimum $F$-weight is said to be {\em $F$-shortest} between $u$ and $v$.  
The $F$-weight of an $F$-shortest path between $u$ and $v$ is referred to as 
the $F$-{\em distance} between $u$ and $v$, 
and is denoted by $\distgtf{G}{T}{F}{u}{v}$. 
\end{definition}

\begin{fact}[see Seb\"o~\cite{DBLP:journals/jct/Sebo90}] \label{fact:dist} 
Let $G$ be a graft, and let $F \subseteq E(G)$ be a minimum join of $(G, T)$. 
Then, for every $x, y\in V(G)$, 
$\distgtf{G}{T}{F}{x}{y} = \nu(G, T) - \nu(G, T\Delta \{x, y\})$.  
\end{fact}

Fact~\ref{fact:dist} implies that the $F$-distance between two vertices does not depend on the minimum join $F$. 
Hence, under Fact~\ref{fact:dist}, 
we sometimes denote $\distgtf{G}{T}{F}{x}{y}$ by $\distgt{G}{T}{x}{y}$. 
That is, 
for a graft $(G, T)$ and vertices $x, y\in V(G)$, 
$\distgt{G}{T}{x}{y}$ denotes $\distgtf{G}{T}{F}{x}{y}$, where $F$ is a minimum join of $(G, T)$.

\begin{definition} 
Let $(G, T)$ be a  graft, and let $r\in V(G)$. Let $F$ be a minimum join. 
We denote the set $\{ x\in V(G): \distgtf{G}{T}{F}{r}{x} = i \}$ by $\levelgtr{G}{T}{r}{i}$  
and the set $\{ x\in V(G): \distgtf{G}{T}{F}{r}{x} < i \}$ by $\laygtr{G}{T}{r}{i}$.  
We also denote the set $\levelgtr{G}{T}{r}{i} \cup \laygtr{G}{T}{r}{i}$ by $\laylegtr{G}{T}{r}{i}$.

For each $i \in \mathbb{Z}$, we denote $\conng{G}{\layler{i}{r}}$ by $\laycompgtr{G}{T}{r}{i}$. 
The set 
\begin{quote} 
$\bigcup \{ \laycompgtr{G}{T}{r}{i} : \min_{x\in V(G)} \distgtf{G}{T}{F}{r}{x} \le i \le \max_{x\in V(G)} \distgtf{G}{T}{F}{r}{x}  \}$
\end{quote}  is denoted by $\laycompall{G}{T}{r}$. 
The members of $\laycompall{G}{T}{r}$ are called the {\em distance components} of $(G, T)$ with respect to $r$. 
\end{definition}

\begin{definition} 
Let $(G, T)$ be a graft, and let $r\in V(G)$. 
Let $C \in \laycompall{G}{T}{r}$. 
We call distance component $C$ a {\em capital}  component if $r$ is a vertex of $C$. 
The capital component that is a member of $\laycompgtr{G}{T}{r}{0}$ 
is called the {\em initial}  component and is denoted by $\initialgtr{G}{T}{r}$.  
We denote the set $V(\initialgtr{G}{T}{r})\cap \levelr{0}{r}$ by $\agtr{G}{T}{r}$,  
the set $V(\initialgtr{G}{T}{r})\setminus \agtr{G}{T}{r}$ by $\dgtr{G}{T}{r}$, 
and the set $V(G)\setminus \agtr{G}{T}{r}\setminus \dgtr{G}{T}{r}$ by $\cgtr{G}{T}{r}$. 
\end{definition}

\section{Basic Properties on Minimum Joins}  

\begin{lemma}[see Seb\"o~\cite{DBLP:journals/jct/Sebo90}] \label{lem:circuit}  
Let $(G, T)$ be graft, and let $F$ be a minimum join. 
If $C$ is a circuit with $w_F(C) = 0$, then $F\Delta E(C)$ is also a minimum join of $(G, T)$. 
Accordingly, every edge of $C$ is allowed. 
\end{lemma}

\begin{lemma}[see Seb\"o~\cite{DBLP:journals/jct/Sebo90}] \label{lem:minimumjoin}  
Let $(G, T)$ be graft, and let $F\subseteq E(G)$. 
Then, $F$ is a minimum join of $(G, T)$ 
if and only if $w_F(C) \ge 0$ for every circuit $C$ of $G$. 
\end{lemma}

\begin{lemma}[see Kita~\cite{kita2017parity}] \label{lem:tnonpositive} 
Let $(G, T)$ be a factor-connected graft, and let $F$ be a minimum join of $(G, T)$.  
Then, $\distgtf{G}{T}{F}{x}{y} \le 0$ holds for every $x, y\in V(G)$. 
\end{lemma}

\section{Seb\"o's Distance Decomposition Theorem}

\begin{theorem}[Seb\"o~\cite{DBLP:journals/jct/Sebo90}] \label{thm:neigh2dist} 
Let $(G, T)$ be a bipartite graft and  $F$ be a minimum join, and let $r\in V(G)$. 
If $x, y\in V(G)$ are adjacent in $G$, then $| \distgtf{G}{T}{F}{r}{u} - \distgtf{G}{T}{F}{r}{v}| = 1$. 
\end{theorem}

\begin{theorem}[Seb\"o~\cite{DBLP:journals/jct/Sebo90}] \label{thm:sebocut} 
Let $(G, T)$ be a bipartite graft, and let $r\in V(G)$. Let $F$ be a minimum join. 
Let $K \in \laycompall{G}{T}{r}$. 
\begin{rmenum} 
\item If $r\not\in \parcut{G}{K}$ holds, then $|\parcut{G}{K}\cap F| =1$. 
\item If $r\in \parcut{G}{K}$ holds, then $\parcut{G}{K}\cap F = \emptyset$.  
\end{rmenum} 
\end{theorem}

\begin{theorem}[Seb\"o~\cite{DBLP:journals/jct/Sebo90}] \label{thm:path2cut} 
Let $(G, T)$ be a bipartite graft, and let $r\in V(G)$. Let $F$ be a minimum join. 
Let $K \in \laycompall{G}{T}{r}$. 
Let $x\in V(G)$, and let $P$ be an $F$-shortest path between $r$ and $x$ in $(G, T)$. 
\begin{rmenum} 
\item If $r, x\in V(K)$ holds, then $E(P)\subseteq E(K)$ holds. 
\item If $r\in V(K)$ and $x\not\in V(K)$ hold, then $|E(P)\cap \parcut{G}{K}| = |E(P)\cap \parcut{G}{K} \cap F| =1$. 
\item If $r\not\in V(K)$ and $x\not\in V(K)$ hold, then either $E(P)\cap \parcut{G}{K} = \emptyset$ holds, or 
 $|E(P)\cap \parcut{G}{K} \cap F| = 1$ and $|  E(P)\cap \parcut{G}{K}  \setminus F| = 1$ hold. 
\item If $r\not\in V(K)$ and $x\in V(K)$ hold, then $|E(P)\cap \parcut{G}{K}| = |E(P)\cap \parcut{G}{K} \cap F| =1$. 
\end{rmenum} 
\end{theorem}

\begin{definition} 
Let $(G, T)$ be a graft, and let $r\in V(G)$. 
We say that $(G, T)$ is {\em primal with respect to} $r$ 
if $\distgt{G}{T}{r}{x} \le 0$ for every $x\in V(G)$. 
\end{definition}

\begin{theorem} \label{thm:seboprimal} 
Let $(G, T)$ be a bipartite graft, let $F$ be a minimum join of $(G, T)$, and let $r\in V(G)$.  
Let $K \in \laycompall{G}{T}{r}$ with $r \not\in V(K)$, 
and let $r_K \in \partial_G(F)\cap V(K)$. 
Then, $F\cap E(K)$ is a minimum join of $(G, T)_F[K]$, 
$\distgtf{G}{T}{F}{r}{x} = \distgtf{G}{T}{F}{r}{r_K} + \distf{ (G, T)_F[K]}{ F\cap E(K)}{r_K}{x}$  
for every $x\in V(K)$, 
and 
$(G, T)_F[K]$ is a primal graft with respect to $r_K$. 
\end{theorem}

\section{General Kotzig-Lov\'asz Decomposition for Grafts}

\begin{definition} 
Let $(G, T)$ be a graft. 
For $u, v\in V(G)$, 
we say $u\gtsim{G}{T} v$ if $u$ and $v$ are contained in the same factor-component, and $\nu(G, T) - \nu(G, T\Delta \{u, v\}) = 0$. 
\end{definition}

\begin{theorem}[Kita~\cite{kita2017parity}] \label{thm:tkl}
If $(G, T)$ is a graft, then $\gtsim{G}{T}$ is an equivalence relation over $V(G)$. 
\end{theorem}

Under Theorem~\ref{thm:tkl}, 
we denote by $\tpart{G}{T}$ the family of equivalence classes of $\gtsim{G}{T}$. 
This family is called the {\em general Kotzig-Lov\'asz decomposition} of a graft.

\section{Extremality in Distance Components}

In this section, 
we provide and prove Lemmas~\ref{lem:inext} and \ref{lem:noear2gext} and then derive Lemma~\ref{lem:exext}. 
Lemma~\ref{lem:exext} then immediately implies Lemma~\ref{lem:ear2posi}.

\begin{definition} 
Let $(G, T)$ be a graft, and let $F$ be a minimum join. 
We say that a set $X\subseteq V(G)$ is {\em extreme} if 
$\distgtf{G}{T}{F}{x}{y} \ge 0$ for every $x, y\in V(G)$. 
\end{definition}

\begin{lemma}\label{lem:inext} 
Let $(G, T)$ be a bipartite graft, let $r\in V(G)$, and let $F$ be a minimum join of $(G, T)$. 
Let $K \in \laycompall{G}{T}{r}$, and let $i$ be the index with $K \in \laycompgtr{G}{T}{r}{i}$. 
Then, $V(K)\cap \levelgtr{G}{T}{r}{i}$ is an extreme set in the graft $(G, T)_F[K]$. 
\end{lemma}  
\begin{proof} 
We prove the lemma by the induction on $i$. 
If $i =  \min_{ v \in V(G) } \distgtf{G}{T}{F}{r}{v}$, then $| V(K) | = 1$, 
and the statement trivially holds. 
Next, let $i > \min_{ v \in V(G) } \distgtf{G}{T}{F}{r}{v}$, and 
assume that the statement holds for every case with smaller $i$.  
Let $x, y \in V(K)\cap \levelgtr{G}{T}{r}{i}$, and let $Q$ be an $F$-shortest path in $(G, T)_F[K]$ between $x$ and $y$. 
Because Theorem~\ref{thm:neigh2dist} implies that $\levelgtr{G}{T}{r}{i}$ is stable, 
we have $E(Q) \subseteq \bigcup \{ \parcut{G}{L} \cup E(L) : L \in \conn{K - \levelgtr{G}{T}{r}{i}} \}$.

Let $L \in \conn{ K - \levelgtr{G}{T}{r}{i} }$. 
The set $E(Q)\cap \parcut{G}{L}$ has an even number of edges, 
among which at most one can be from $F$, according to Theorem~\ref{thm:sebocut}. 
Hence, we have $w_F( Q. \parcut{G}{L} ) \ge 0$.  

We next prove $w_F(Q. E(L)) \ge 0$. 
Each connected component of $Q. E(L)$ is a path in $L$ between two vertices from $V(L)\cap \levelgtr{G}{T}{r}{i-1}$. 
The induction hypothesis implies that $ V(L)\cap \levelgtr{G}{T}{r}{i-1}$ is an extreme set of $(G, T)_F[L]$. 
Additionally, Lemma~\ref{lem:minimumjoin}  implies that $F\cap E(L)$ is a minimum join of $(G, T)_F[L]$. 
Hence, every connected component of $  Q. E(L) $ has a nonnegative $F$-weight. 
Thus, $w_F(Q. E(L)) \ge 0$ follows.

Therefore, we have $w_F( Q. \parcut{G}{L} \cup E(L) ) \ge 0$. 
This implies $w_F(Q) \ge 0$. 
The lemma is proved.

\end{proof}

\begin{lemma}\label{lem:noear2gext} 
Let $(G, T)$ be a bipartite graft, let $r\in V(G)$, and let $F$ be a minimum join of $(G, T)$. 
Let $K$ be a member of $\laycompall{G}{T}{r}$, and let $i$ be the index with $K \in \laycompgtr{G}{T}{r}{i}$. 
Assume that every round ear path relative to $K$ has a $F$-weight that is no less than $0$. 
Then, $V(K)\cap \levelgtr{G}{T}{r}{i}$ is an extreme set in $(G, T)$. 
\end{lemma} 
\begin{proof} 
Let $P$ be a path between two vertices $u,v \in V(K)\cap \levelgtr{G}{T}{r}{i}$.  
Note that $P$ is the sum of the members from $\connp{P. E(K)}$ and $\connp{P - E(K)}$. 
The assumption implies that every member of $\connp{P - E(K)}$ has an  $F$-weight no less than $0$. 
Lemma~\ref{lem:inext} implies that every member of $\connp{P. E(K)}$  has an $F$-weight no less than $0$. 
Hence, $w_F(P) \ge 0$ follows. 
This proves the lemma. 
\end{proof}

\begin{lemma} \label{lem:exext} 
Let $(G, T)$ be a bipartite graft, let $r\in V(G)$, and let $F$ be a minimum join of $(G, T)$. 
Let $K$ be a member of $\laycompall{G}{T}{r}$,  and let $i$ be the index with $K \in \laycompgtr{G}{T}{r}{i}$. 
\begin{rmenum} 
\item \label{item:noear} Let $P$ be a round ear path relative to $K$. Then, $w_F(P) \ge 0$ holds.  
If $w_F(P) = 0$ holds, then $r\not\in V(K)$ holds, 
and either bond of $P$ is  $r_K$, where $r_K$ is the vertex from $\partial_{G}(F)\cap V(K)$. 
\item \label{item:ext} 
 $V(K)\cap \levelgtr{G}{T}{r}{i}$ is an extreme set of the graft $(G, T)$. 
\end{rmenum} 
\end{lemma} 
\begin{proof} 
We prove the lemma by the induction on $i$. 
First, let $i = \max_{x \in V(G)} \distgtf{G}{T}{F}{r}{x}$, that is, $K = G$.   
The statement~\ref{item:noear} clearly holds, and Lemma~\ref{lem:inext} proves \ref{item:ext}.

Next, let $i < \max_{x \in V(G)} \distgtf{G}{T}{F}{r}{x}$, and assume that \ref{item:noear} and \ref{item:ext} hold for every case with greater $i$. 
Let $u, v\in V(P)$ be the bonds of $P$.  
For each $\alpha \in \{u, v\}$, let $e_\alpha \in E(P)$ be the edge of $P$ that is connected to $\alpha$, 
and let $z_\alpha\in V(P)$ be the end of $e_\alpha$ that is distinct from $\alpha$. 
Then, Theorem~\ref{thm:neigh2dist} implies $z_u, z_v \in V(\hat{K})\cap \levelgtr{G}{T}{r}{i+1}$, where $\hat{K}$ is the member of $\laycompgtr{G}{T}{r}{i+1}$ with $V(K)\subseteq V(\hat{K})$. 
From the induction hypothesis on \ref{item:ext}, we have $w_F(z_uPz_v) \ge 0$. 
Theorem~\ref{thm:sebocut} implies $w_F(e_u) + w_F(e_v) \ge 0$; 
the equality holds if and only if either $e_u$ or $e_v$ is in $F$.  
Hence, we obtain $w_F(P) \ge 0$; 
the equality holds if and only if $w_F(z_uPz_v) = 0$ and either $e_u$ or $e_v$ is in $F$.  
This implies \ref{item:noear}. 
Hence, Lemma~\ref{lem:noear2gext} now proves \ref{item:ext}.  
This completes the proof. 
\end{proof}

Theorem~\ref{thm:sebocut} and Lemma~\ref{lem:exext} imply Lemma~\ref{lem:ear2posi}.

\begin{lemma} \label{lem:ear2posi} 
Let $(G, T)$ be a bipartite graft, let $r\in V(G)$, and let $F$ be a minimum join of $(G, T)$. 
Let $K$ be a member of $\laycompall{G}{T}{r}$ with $r\in V(G)$. 
Then,  $w_F(P) \ge 2$ holds for every round ear path $P$ relative to $K$. 
\end{lemma}

\section{Fundamental Unit for Distances}

We provide and prove Lemmas~\ref{lem:ar4ad} and \ref{lem:ar4c} and then derive Lemma~\ref{lem:ar}. 
We then use Lemma~\ref{lem:ar} to derive Theorem~\ref{thm:unit4dist}. 

\begin{lemma} \label{lem:ar4ad} 
Let $(G, T)$ be a primal bipartite graft with respect to $r\in V(G)$, and let $F$ be a minimum join of $(G, T)$.  
Then, for every $x\in \agtr{G}{T}{r}$ and every $y\in V(G)$, 
$\distgtf{G}{T}{F}{r}{y} \le \distgtf{G}{T}{F}{x}{y}$. 
\end{lemma} 
\begin{proof} 
Let $l := \min_{ z \in V(G) } \distgtf{G}{T}{F}{r}{z} $. 
We prove the lemma by induction on $l$. 
If $l = 0$, then $V(G) = \{r\}$, and the statement trivially holds.

Next, let $l \le -1$, and assume that the statement holds for every case where $l$ is smaller.  
Let $x \in \agtr{G}{T}{r}$ and $y \in V(G)$, and let $Q$ be an $F$-shortest path between $x$ and $y$. 
Let $\mathcal{K} := \{ K \in \conng{G}{ \dgtr{G}{T}{r} }: V(K) \cap V(Q) \neq \emptyset \}$. 
Because $\agtr{G}{T}{r}$ is stable, we have $E(Q) \subseteq \bigcup_{K \in \mathcal{K}} \parcut{G}{K} \cup E(K) $.

\begin{pclaim} \label{claim:each} 
If $K\in\mathcal{K}$ satisfies $y \not\in V(K)$, then $w_F( Q. E(K)) \ge 0$ and $w_F( Q. \parcut{G}{K} ) \ge 0$.  
If $K\in\mathcal{K}$ satisfies $y \in V(K)$, 
then $w_F( Q. E(K)) \ge  \distf{ (G, T)_F[K] }{F \cap E(K)}{r_K}{y}$ and $w_F( Q. \parcut{G}{K} ) \ge -1$, 
where $r_K$ is the vertex from $\partial_{G}(F) \cap V(K)$.   
\end{pclaim} 
\begin{proof} 
Let $K \in \mathcal{K}$. 
First, consider the case  $y \in V(K)$. 
Theorem~\ref{thm:sebocut} proves  that  $w_F( Q. \parcut{G}{K} ) \ge -1$. 
According to Theorem~\ref{thm:seboprimal},  
  the graft $(G, T)_F[K]$ is primal with respect to $r_K$, 
  in which $\min_{ z \in V(K) } \distf{ (G, T)_F[K]}{F \cap E(K)}{r_K}{z} = l + 1$ 
  and $F \cap E(K)$ is a minimum join.  
Hence, the induction hypothesis can be applied for $(G, T)_F[K]$. 
If $R$ is a member of $\conn{ Q. E(K) }$ with $y\not\in V(R)$, 
then Theorem~\ref{thm:neigh2dist} implies that $R$ is a path between two vertices from $\apr{ (G, T)_F[K] }{r_K}$. 
Hence, $w_F(R) \ge 0$ holds.  
Otherwise, that is, if $y\in V(R)$ holds,  
then $R$ is a path between $y$ and a vertex from $\apr{ (G, T)_F[K] }{r_K}$. 
Hence, $w_F(R) \ge \distf{ (G, T)_F[K] }{F\cap E(K)}{r_K}{y}$  holds. 
Thus, $w_F(Q. E(K)) \ge \distf{ (G, T)_F[K] }{F\cap E(K)}{r_K}{y}$ is obtained.  
The case $y \not\in V(K)$ can be proved in the same way. 
\end{proof}

First, consider the case $y \in \dgtr{G}{T}{r}$; let $K$ be the member of $\mathcal{K}$ with $y\in V(K)$.
 Claim~\ref{claim:each} implies that 
 $w_F(Q) \ge \distf{ (G, T)_F[K] }{F \cap E(K)}{r_K}{y} - 1$. 
Additionally, Theorem~\ref{thm:seboprimal} implies $\distf{ (G, T)_F[K] }{F \cap E(K)}{r_K}{y} = \distgtf{G}{T}{F}{r}{y} + 1$.  
Thus, we obtain $\distgtf{G}{T}{F}{r}{y} \le \distgtf{G}{T}{F}{x}{y}$.   
The case $y\in \agtr{G}{T}{r}$ can be proved in the same way.  
This completes the proof. 
\end{proof}

\begin{lemma} \label{lem:ar4c} 
Let $(G, T)$ be a bipartite graft,  let $r\in V(G)$, and let $F$ be a minimum join of $(G, T)$.  
Then, for every $x\in \agtr{G}{T}{r}$ and every $y\in \cgtr{G}{T}{r}$, 
$\distgtf{G}{T}{F}{r}{y} \le \distgtf{G}{T}{F}{x}{y}$.  
\end{lemma} 
\begin{proof} 
Let $x\in \agtr{G}{T}{r}$ and  $y\in \cgtr{G}{T}{r}$, 
and let $P$ be an $F$-shortest path between $x$ and $y$. 
Trace $P$ from $y$, and let $z$ be the first encountered vertex that is in $\agtr{G}{T}{r}$. 
Because Lemma~\ref{lem:exext} implies that $\agtr{G}{T}{r}$ is extreme, we have $w_F(xPz) \ge \distgtf{G}{T}{F}{x}{z} \ge 0$. 
Let $Q$ be an $F$-shortest path between $r$ and $z$.  Thus, we have $w_F(xPz) \ge w_F(Q)$.

Theorem~\ref{thm:path2cut} implies $V(Q) \subseteq \agtr{G}{T}{r}\cup \dgtr{G}{T}{r}$. 
Therefore, $yPz + Q$ is a path between $r$ and $y$ 
for which $w_F(Q + zPy) \le w_F(xPz + zPy) = w_F(P)$ holds. 
Accordingly, $\distgtf{G}{T}{F}{r}{y} \le \distgtf{G}{T}{F}{x}{y}$ follows. 
This proves the lemma. 
\end{proof}

\begin{lemma} \label{lem:ar} 
Let $(G, T)$ be a bipartite graft,  let $r\in V(G)$, and let $F$ be a minimum join of $(G, T)$.  
Then, for every $x\in \agtr{G}{T}{r}$ and every $y\in V(G)$, 
$\distgtf{G}{T}{F}{r}{y} \le \distgtf{G}{T}{F}{x}{y}$. 
\end{lemma} 
\begin{proof} 
For $y\in \cgtr{G}{T}{r}$, Lemma~\ref{lem:ar4c} implies the claim. 
Let $y\in \agtr{G}{T}{r}\cup \dgtr{G}{T}{r}$. 
According to Theorem~\ref{thm:path2cut}, 
we can consider the claim for $(G, T)_F[\initialgtr{G}{T}{r}]$ in stead of $(G, T)$; that is, 
we can assume that  $(G, T)$ is primal with respect to $r$.  
Lemma~\ref{lem:ar4ad} then implies the claim. 
This proves the lemma. 
\end{proof}

\begin{lemma}[Kita~\cite{kita2017parity}] \label{lem:root2partition}  
Let $(G, T)$ be a bipartite graft. 
Let $r\in V(G)$, and let $S$ be the member of  $\tpart{G}{T}$ with $r\in S$. 
Then, 
$S \subseteq \agtr{G}{T}{r}$ holds. 
\end{lemma}

Lemmas~\ref{lem:ar} and \ref{lem:root2partition} imply Theorem~\ref{thm:unit4dist}.

\begin{theorem} \label{thm:unit4dist} 
Let $(G, T)$ be a bipartite graft, and let $F$ be a minimum join of $(G, T)$. 
Let $x, y\in V(G)$ be vertices with $x\gtsim{G}{T} y$. 
Then,  $\distgtf{G}{T}{F}{x}{z} = \distgtf{G}{T}{F}{y}{z}$ for every $z\in V(G)$. 
\end{theorem} 
\begin{proof} 
Lemma~\ref{lem:root2partition} implies $y \in \agtr{G}{T}{x}$. 
Hence, Lemma~\ref{lem:ar} implies $\distgtf{G}{T}{F}{x}{z} \le \distgtf{G}{T}{F}{y}{z}$ for every $z\in V(G)$. 
Similarly, 
we can also obtain $x \in \agtr{G}{T}{y}$ and $\distgtf{G}{T}{F}{x}{z} \ge \distgtf{G}{T}{F}{y}{z}$ for every $z\in V(G)$.  
Thus, 
 for every $z\in V(G)$,  $\distgtf{G}{T}{F}{x}{z} = \distgtf{G}{T}{F}{y}{z}$.  
\end{proof}

\section{Structure of Initial Components} 
\subsection{Structure of $\agtr{G}{T}{r}$}

\begin{lemma} \label{lem:fcomp2ar} 
Let $(G, T)$ be a bipartite graft, and let $r\in V(G)$. 
Let $x\in \agtr{G}{T}{r}$, let $S$ be the member of $\tpart{G}{T}$ with $x\in S$, 
and let $C\in\tcomp{G}{T}$ be a factor-component with $x\in V(C)$.  
Then, $\agtr{G}{T}{r} \cap V(C) = S$, and $V(C)\setminus S \subseteq \dgtr{G}{T}{r}$.  
\end{lemma} 
\begin{proof} 
Because Theorem~\ref{thm:sebocut} implies that $\parcut{G}{\initialgtr{G}{T}{r}}$ contains no allowed edges, 
we have $V(C)\subseteq \agtr{G}{T}{r} \cup \dgtr{G}{T}{r}$. 
From Theorem~\ref{thm:unit4dist}, we 
have $\distgt{G}{T}{r}{x} = \distgt{G}{T}{r}{y} = 0$ for every $y\in S$; 
hence,  $S \subseteq \agtr{G}{T}{r}$ holds.  
Additionally, Lemmas \ref{lem:tnonpositive} and \ref{lem:exext} imply $( V(C)\setminus S )\cap \agtr{G}{T}{r} = \emptyset$. 
Hence, $V(C)\setminus S \subseteq \dgtr{G}{T}{r}$ follows.   
This completes the proof. 
\end{proof}

Lemma~\ref{lem:fcomp2ar} easily implies Lemma~\ref{lem:ar2partition}. 

\begin{lemma} \label{lem:ar2partition} 
Let $(G, T)$ be a bipartite graft, and let $r\in V(G)$. 
Then, there are equivalence classes $S_1, \ldots, S_k \in\tpart{G}{T}$, where $k\ge 1$, 
such that $\agtr{G}{T}{r} = S_1 \cup \cdots \cup S_k$. 
\end{lemma}

\subsection{Structure of $\dgtr{G}{T}{r}$ } 
\subsubsection{Classification of  Odd Components}

\begin{definition} 
Let $(G, T)$ be a bipartite graft,  and let $r\in V(G)$. 
For each $S \in \tpart{G}{T}$ with $S \subseteq \agtr{G}{T}{r}$,  
we denote by $\neicompgtr{G}{T}{r}{S}$ the set of connected components of $G[ \dgtr{G}{T}{r} ]$ that are adjacent to $S$ with allowed edges. 
The set of vertices from the members of $\neicompgtr{G}{T}{r}{S}$ is denoted by $\neisetgtr{G}{T}{r}{S}$.   
\end{definition}

\begin{lemma} \label{lem:neigh2elem} 
Let $(G, T)$ be a bipartite graft, let $r\in V(G)$, and let $F$ be a minimum join. 
Let $K$ be a member of $\laycompall{G}{T}{r}$ with $r\not\in V(K)$.  
Let $e$ be the edge from $\parcut{G}{K} \cap F$, and let $r_K$ be the end of $e$ with $r_K \in V(K)$. 
Let $f \in \parcut{G}{K}$ be an allowed edge, and let $t$ be the end of $f$ with $t\in V(K)$. 
Then,  
the graft $(G, T)_F[K]$ has a path between $r_K$ and $t$ whose $F$-weight is $0$ and the edges are allowed in $(G, T)$. 
\end{lemma} 
\begin{proof} 
Let $Q$ be an $F$-shortest path between $r_K$ and $t$ in $(G, T)_F[K]$. Theorem~\ref{thm:seboprimal} implies $w_F(Q) = 0$. 
Let $F'$ be a minimum join of $(G, T)$ with $f \in F'$. 
Then, $(F'\setminus E(K)) \cup ( (F\cap E(K)) \Delta E(Q) )$ is a minimum join of $(G, T)$ 
 that contains  $E(Q)\setminus F$. 
 Therefore, every edge of $Q$ is allowed, and thus $Q$ is a desired path. 
\end{proof}

\begin{lemma} \label{lem:neigh2sole}
Let $(G, T)$ be a bipartite graft, and let $r\in V(G)$. 
If $S_1$ and $S_2$ are distinct members of $\tpart{G}{T}$ with $S_1 \cup S_2 \subseteq \agtr{G}{T}{r}$, 
then $\neicompgtr{G}{T}{r}{S_1} \cap \neicompgtr{G}{T}{r}{S_2} = \emptyset$. 
\end{lemma} 
\begin{proof} 
Suppose that there exists $K \in \neicompgtr{G}{T}{r}{S_1} \cap \neicompgtr{G}{T}{r}{S_2}$. 
Lemma~\ref{lem:neigh2elem} implies that $S_1$ and $S_2$ are contained in the same factor-component. 
This contradicts Lemma~\ref{lem:fcomp2ar}. 
The lemma is proved. 
\end{proof}

\subsubsection{Definition of Critical Sets} 

We define two concepts called negative and critical sets. 
Let $(G, T)$ be a graft, let $F$ be a minimum join of $(G, T)$, and let $S\in \gtpart{G}{T}$. 

\begin{definition} 
We call $X\subseteq V(G)\setminus S$ an {\em $F$-negative set} for $S$ 
if, for every $x\in X$, 
there exists a vertex $y \in S$ such that 
there is a path between $x$ and $y$ with negative $F$-weight whose vertices except $y$ are contained in $X$. 
\end{definition}

Lemma~\ref{lem:neg2cr} can easily be confirmed. 

\begin{lemma} \label{lem:neg2cr} 
Let $(G, T)$ be a bipartite graft, let $F$ be a minimum join, and let $S\in \gtpart{G}{T}$.  
Then, there exists the maximum $F$-negative set for $S$.  
\end{lemma}

\begin{definition} 
Under Lemma~\ref{lem:neg2cr},  
we call the maximum $F$-negative set the {\em $F$-critical set} for $S$ and denoted this set by $\coupgtf{G}{T}{F}{S}$.  
\end{definition}

\subsubsection{Critical Sets and Odd Components}

We provide and prove Lemmas~\ref{lem:neiset2coup}, \ref{lem:nonallowed2dist}, and \ref{lem:coup2neiset}  and then derive Lemma~\ref{lem:dr2char}.

\begin{lemma} \label{lem:neiset2coup} 
Let $(G, T)$ be a bipartite graft, and let $r\in V(G)$.  
Let $S$ be a member of $\tpart{G}{T}$ with $S \subseteq \agtr{G}{T}{r}$.  
Then,  $\neisetgtr{G}{T}{r}{S} \subseteq \coupgtf{G}{T}{F}{S}$ holds. 
\end{lemma} 
\begin{proof} 
Let $F$ be a minimum join of $(G, T)$. 
Let $K \in \neicompgtr{G}{T}{r}{S}$. 
Under Theorem~\ref{thm:sebocut}, let $e \in \parcut{G}{K} \cap F$ and $r_K \in \partial_G(e)\cap V(K)$. 
Theorem~\ref{thm:seboprimal} implies that the graft $(G, T)_F[K]$ is  primal with respect to $r_K$, 
for which $F \cap E(K)$ is a minimum join.  
Hence, for every $x\in V(K)$, 
there is a path $P$ between $x$ and  $r_K$ with $w_F(P) \le 0$ and $V(P) \subseteq V(K)$. 
 Lemma~\ref{lem:neigh2sole} implies that 
 $e$ joins $r_K$ and a vertex $s \in S$. 
Therefore, $P+ e$ is a path between $x$ and $s$ with 
$w_F(P + e) < 0$ and $V(P)\setminus \{s\} \subseteq V(K)$. 
Thus, we obtain $V(K)\subseteq \coupgtf{G}{T}{F}{S}$. 
This proves the lemma. 
\end{proof}

\begin{lemma}  \label{lem:nonallowed2dist} 
Let $(G, T)$ be a bipartite graft. 
Let $e$ be an edge between distinct vertices $x$ and $y$.  
If $e$ is not allowed, then $\distgt{G}{T}{x}{y} = 1$. 
\end{lemma} 
\begin{proof} 
Assume $\distgt{G}{T}{x}{y} \neq 1$.  Then, Theorem~\ref{thm:neigh2dist} implies $\distgt{G}{T}{x}{y} = -1$.   
Let $F$ be a minimum join of $(G, T)$, 
and let $P$ be an $F$-shortest path between $x$ and $y$. 
Clearly, $e \not\in F$ and $e \not\in E(P)$ hold. 
Hence, $P + e$ is a circuit whose $F$-weight is $0$. 
This implies from Lemma~\ref{lem:circuit} that $e$ is allowed. 
The lemma is proved. 
\end{proof}

\begin{lemma} \label{lem:coup2neiset} 
Let $(G, T)$ be a bipartite graft, and let $r\in V(G)$.  Let $F$ be a minimum join of $(G, T)$. 
Let $S$ be a member of $\tpart{G}{T}$ with $S \subseteq \agtr{G}{T}{r}$. 
Then, $\coupgtf{G}{T}{F}{S} \subseteq \neisetgtr{G}{T}{r}{S}$ holds. 
\end{lemma} 
\begin{proof} 

The following claim is easily implied from Lemma~\ref{lem:exext}. 

\begin{pclaim} \label{claim:ar2disjoint}  
$\coupgtf{G}{T}{F}{S} \cap \agtr{G}{T}{r} = \emptyset$.  
\end{pclaim}

The next claim can be proved from Lemma~\ref{lem:ar}.

\begin{pclaim} \label{claim:cr2disjoint} 
$\coupgtf{G}{T}{F}{S} \cap  \cgtr{G}{T}{r} = \emptyset$. 
\end{pclaim} 
\begin{proof} 
Let $x\in \parNei{G}{\initialgtr{G}{T}{r}}$. 
Theorem~\ref{thm:neigh2dist} implies  $\distgtf{G}{T}{F}{r}{x} = 1$.   
Lemma~\ref{lem:ar} further implies $\distgtf{G}{T}{F}{s}{x} = 1$ for every $s \in S$.        
Therefore, $x \not\in \coupgtf{G}{T}{F}{S}$ follows. 
Accordingly, $\parNei{G}{\initialgtr{G}{T}{r}} \cap \coupgtf{G}{T}{F}{S} = \emptyset$. 
This implies the claim. 
\end{proof} 

From Claims~\ref{claim:ar2disjoint} and \ref{claim:cr2disjoint}, we now have $\coupgtf{G}{T}{F}{S} \subseteq \dgtr{G}{T}{r}$. 
The next claim is implied from Theorem~\ref{thm:unit4dist} and Lemma~\ref{lem:nonallowed2dist}. 

\begin{pclaim} \label{claim:dr2disjoint} 
$\dgtr{G}{T}{r}\setminus \neisetgtr{G}{T}{r}{S} = \emptyset$. 
\end{pclaim} 
\begin{proof} 
Let $K \in \conng{G}{\dgtr{G}{T}{r}} \setminus \neicompgtr{G}{T}{r}{S}$. 
Note that every path between vertices in $S$ and $V(K)$ must contain a vertex from the set $( \parNei{G}{K} \setminus S)  \cup (\parNei{G}{S}\cap V(K))$. 
We prove the claim by proving that this set is disjoint from $\coupgtf{G}{T}{F}{S}$.  

Claim~\ref{claim:ar2disjoint} implies  $(\parNei{G}{K} \setminus S) \cap \coupgtf{G}{T}{F}{S} = \emptyset$. 
Next, let $y \in \parNei{G}{S}\cap V(K)$, and let $x\in S$ be the vertex with $xy \in E(G)$. 
The edge $xy$ is not allowed; 
hence, Lemma~\ref{lem:nonallowed2dist} implies $\distgtf{G}{T}{F}{x}{y} > 0$. 
Theorem~\ref{thm:unit4dist} further implies $y \not\in \coupgtf{G}{T}{F}{S}$.  
This implies $\parNei{G}{S}\cap V(K)  \cap \coupgtf{G}{T}{F}{S} = \emptyset$. 
Thus, we obtain the claim. 
\end{proof} 

From Claims~\ref{claim:ar2disjoint}, \ref{claim:cr2disjoint}, and \ref{claim:dr2disjoint}, 
the lemma is proved.  
\end{proof}

Lemmas~\ref{lem:neiset2coup} and \ref{lem:coup2neiset} imply Lemma~\ref{lem:dr2char}.

\begin{lemma} \label{lem:dr2char} 
Let $(G, T)$ be a bipartite graft, and let $r\in V(G)$. 
Let $S$ be a member of $\tpart{G}{T}$ with $S \subseteq \agtr{G}{T}{r}$. 
Then, $\neisetgtr{G}{T}{r}{S} = \coupgtf{G}{T}{F}{S}$.   
\end{lemma}

From Lemma~\ref{lem:dr2char},  
it can be observed that $\neisetgtr{G}{T}{r}{S}$ and $\coupgtf{G}{T}{F}{S}$ do not depend on the choice of  root $r$ or  minimum join $F$. 
Therefore, 
hereinafter 
we simply denote $\neicompgtr{G}{T}{r}{S}$ and $\neisetgtr{G}{T}{r}{S}$ by $\neicompgt{G}{T}{S}$ and $\neisetgt{G}{T}{S}$, respectively, 
and denote $\coupgtf{G}{T}{F}{S}$ by $\coupgt{G}{T}{S}$.

\subsection{Theorem for Initial Components}

Lemmas~\ref{lem:ar2partition} and \ref{lem:dr2char} imply the next theorem. 

\begin{theorem} \label{thm:icomp} 
Let $(G, T)$ be a bipartite graft, and let $r\in V(G)$. 
Then, there exist $S_1, \ldots, S_k \in \tpart{G}{T}$, where $k\ge 1$, 
such that 
\begin{rmenum} 
\item $\agtr{G}{T}{r} = S_1 \dot\cup \cdots \dot\cup S_k$, 
\item $\dgtr{G}{T}{r}  = \coupgt{G}{T}{S_1} \dot\cup \cdots \dot\cup  \coupgt{G}{T}{S_k}$,  and  
\item 
$\conng{G}{\dgtr{G}{T}{r}} = \conng{G}{\coupgt{G}{T}{S_1}} \dot\cup \cdots \dot\cup \conng{G}{\coupgt{G}{T}{S_k}}$. 
\end{rmenum} 
\end{theorem}

\section{Rootlization and Root Set Distances} 
\subsection{Rootlization}

\begin{definition} 
Let $(G, T)$ be a graft, and let $X \subseteq V(G)$. 
Let $\hat{G}$ be the graph such that $V(\hat{G}) = V(G) \cup \{r, s\}$, where $r, s\not\in V(G)$,  
and $E(\hat{G}) = E(G) \cup \{rs \} \cup \{ sx : x\in X\}$. 
Let $\hat{T} = T \cup \{r, s\}$. 
We call  the graft $(\hat{G}, \hat{T})$  the {\em rootlization} of $(G, T)$ 
by the {\em mount} $X$,  {\em root} $r$, and  {\em attachment} $s$. 
We also denote $(\hat{G}, \hat{T})$ by $\extend{G}{T}{X}{r}{s}$. 
\end{definition}

\begin{lemma}[Kita~\cite{kita2021bipartite}] \label{lem:extend} 
Let $(G, T)$ be a graft, and let $X \subseteq V(G)$ be an extreme set. 
Let $(\hat{G}, \hat{T})$ be the rootlization of $(G, T)$ by the mount $X$,  root $r$, and  attachment $s$. 
\begin{rmenum} 
\item \label{item:extend:join} A set $\hat{F} \subseteq E(\hat{G})$ is a minimum join if and only if 
$\hat{F}$ is of the form $\hat{F} = F \cup \{rs\}$, where $F$ is a minimum join of $(G, T)$. 
\item \label{item:extend:distance} 
 Let $F$ be a minimum join of $(G, T)$, and let  $\hat{F} = F \cup \{rs\}$. 
Then,  
$\distgtf{\hat{G}}{\hat{T}}{\hat{F}}{r}{y} = \min_{x\in X} \distgtf{G}{T}{F}{x}{y}$ holds for every $y\in V(G)$, 
whereas $\distgtf{\hat{G}}{\hat{T}}{\hat{F}}{r}{s} = -1$. 
Additionally, $P$ is an $\hat{F}$-shortest path in $\hat{G}$ between $r$ and $y$ 
 if and only if $P$ is of the form $P = P' + xs + rs$, where 
 $x\in X$ is a vertex such that  $\distgtf{G}{T}{F}{x}{y}$ is the minimum among every vertex $x\in X$ 
 and $P$ is an $F$-shortest path in $G$ between $x$ and $y$. 
\end{rmenum} 
\end{lemma} 
\begin{proof} 
For proving \ref{item:extend:join}, let $F$ be a minimum join of $(G, T)$, and let $\hat{F} := F \cup \{rs\}$. 
Clearly, $\hat{F}$ is a join of $(\hat{G}, \hat{T})$. 
Let $C$ be a circuit in $\hat{G}$. 
If $s\not\in V(C)$ holds, then $C$ is a circuit of $G$, and Lemma~\ref{lem:minimumjoin} implies $w_{\hat{F}}(C) = w_F(C) \ge 0$. 
Assume $s \in V(C)$. 
Then, $C$ is of the form $P + sx + sy$, 
where $x$ and $y$ are vertices from $X$, whereas $P$ is a path in $G$ between $x$ and $y$. 
Because $X$ is extreme, $w_F(P) \ge 0$ holds. 
Hence, $w_{\hat{F}}(C) \ge 2$. 
Therefore, Lemma~\ref{lem:minimumjoin} implies that $\hat{F}$ is a minimum join of $\hat{G}$. 
Furthermore, Lemma~\ref{lem:circuit} implies that no edge between $s$ and $X$ is allowed. 
By Lemma~\ref{lem:minimumjoin} again, this implies that every minimum join of $(\hat{G}, \hat{T})$ is a union of a minimum join of $(G, T)$ and $\{rs\}$. 
The statement \ref{item:extend:join} is proved. 
The statement \ref{item:extend:distance} can easily be confirmed. 
\end{proof}

\begin{lemma} \label{lem:extend2sim} 
Let $(G, T)$ be a graft,  and let $R \subseteq V(G)$ be an extreme set. 
Let $(\hat{G}, \hat{T}):= \extend{G}{T}{R}{r}{s}$.  
Then, the following properties hold: 
\begin{rmenum} 
\item $\tcomp{\hat{G}}{\hat{T}} = \tcomp{G}{T} \cup \{ \hat{G}.  rs  \}$.   
\item $\distgt{G}{T}{x}{y} \ge \distgt{\hat{G}}{\hat{T}}{x}{y}$ holds for every $x, y\in V(G)$. 
\item Accordingly, for every $x, y\in V(G)$, $x \gtsim{\hat{G}}{\hat{T}} y$ implies $x \gtsim{G}{T} y$. 
\end{rmenum} 
\end{lemma}

\subsection{Root Set Distance}

\begin{definition} 
Let $(G, T)$ be a graft, let $F\subseteq E(G)$, and let $R \subseteq V(G)$ be an extreme set. 
For every $x\in V(G)$, we denote $\min_{r \in R} \distgtf{G}{T}{F}{r}{x}$ by $\setdistgtf{G}{T}{F}{R}{x}$.  
\end{definition}

From Fact~\ref{fact:dist}, 
it can easily be confirmed that $\distgtf{G}{T}{F}{R}{x} = \distgtf{G}{T}{F'}{R}{x}$ for any two minimum joins $F$ and $F'$.   
Therefore,  
we sometimes denote $\distgtf{G}{T}{F}{R}{x}$ as $\distgt{G}{T}{R}{x}$ if $F$ is a minimum join of $(G, T)$.

\begin{definition}
Let $(G, T)$ be a graft, and let $R \subseteq V(G)$ be an extreme set. 
For every $i \in \mathbb{Z}$, 
we denote the set $\{ x\in V(G): \distgt{G}{T}{R}{x} = i \}$ by $\levelgtr{G}{T}{R}{i}$  
and the set $\{ x\in V(G): \distgt{G}{T}{R}{x} < i \}$ by $\laygtr{G}{T}{R}{i}$, 
and the set $\levelgtr{G}{T}{R}{i}  \cup \laygtr{G}{T}{R}{i}$  by $\laylegtr{G}{T}{R}{i}$, 
\end{definition}

\begin{definition}
Let $(G, T)$ be a graft, and let $R \subseteq V(G)$ be an extreme set.  
We call the sum of every member $K$ of $\conng{G}{\laylegtr{G}{T}{R}{0}}$ with $R \cap V(K) \neq \emptyset$ the {\em initial subgraph} of $R$ 
and denote this subgraph by $\initialgtr{G}{T}{R}$. 
We denote the set $V(\initialgtr{G}{T}{R})\cap \levelgtr{G}{T}{R}{0}$ by $\agtr{G}{T}{R}$,  
the set $V(\initialgtr{G}{T}{R})\setminus \agtr{G}{T}{R}$ by $\dgtr{G}{T}{R}$, 
and the set $V(G)\setminus \agtr{G}{T}{R}\setminus \dgtr{G}{T}{R}$ by $\cgtr{G}{T}{R}$. 
\end{definition}

\subsection{Correspondence between Rootlization and Root Set Distance} 

We list some observations on rootlization and distances determined by a set of roots. 
These observations can easily be confirmed from the definitions.  
We use the following properties in the remainder of this paper without explicitly mentioning it.

\begin{observation} 
Let $(G, T)$ be a graft,  and let $R \subseteq V(G)$ be an extreme set. 
Let $(\hat{G}, \hat{T}):= \extend{G}{T}{R}{r}{s}$.  
Let $F\subseteq E(G)$ be a minimum join of $(G, T)$, and let $\hat{F} := F \cup \{rs \}$. 
Then, $\distgtf{\hat{G}}{\hat{T}}{\hat{F}}{r}{x} = \setdistgtf{G}{T}{F}{R}{x}$ holds for every $x\in V(G)$. 
\end{observation}

\begin{observation} 
Let $(G, T)$ be a graft, and let $R \subseteq V(G)$ be an extreme set.  
Let $(\hat{G}, \hat{T}):= \extend{G}{T}{X}{r}{s}$.   
Then, 
\begin{rmenum} 
\item $\laylegtr{G}{T}{R}{i} = \laylegtr{\hat{G}}{\hat{T}}{r}{i}$ for every $i \in \mathbb{Z}$ with $i < -1$, 
\item $\laylegtr{\hat{G}}{\hat{T}}{r}{-1} = \laylegtr{G}{T}{R}{-1}  \dot\cup \{ s \} $, and   
\item $\laylegtr{\hat{G}}{\hat{T}}{r}{i} = \laylegtr{G}{T}{R}{i}  \dot\cup \{ r, s \} $ for every $i \in \mathbb{Z}$ with $i \ge 0$. 
\end{rmenum} 
\end{observation}

\begin{observation} \label{obs:initreduc} 
Let $(G, T)$ be a bipartite graft with color classes $A$ and $B$,  and let $X \subseteq V(G)$ be an extreme set.  
Let $(\hat{G}, \hat{T}):= \extend{G}{T}{X}{r}{s}$.   
Then,  
$\initialgtr{\hat{G}}{\hat{T}}{r} = \initialgtr{G}{T}{X} + \parcut{\hat{G}}{s}$, 
$\agtr{\hat{G}}{\hat{T}}{r} = \agtr{G}{T}{X} \cup \{r\}$, and $\dgtr{\hat{G}}{\hat{T}}{r}  = \dgtr{G}{T}{X} \cup \{s\}$. 
\end{observation}

\section{Initial Subgraph of Homogeneous Extreme Root Set}

\begin{lemma} \label{lem:extend2icomp} 
Let $(G, T)$ be a bipartite graft with color classes $A$ and $B$, and let $X \subseteq A$ be an extreme set. 
Let $(\hat{G}, \hat{T}):= \extend{G}{T}{X}{r}{s}$.   
Let $S$ be a member of $\gtpart{\hat{G}}{\hat{T}}$ with $S \subseteq \agtr{\hat{G}}{\hat{T}}{r}$ and $S \neq \{r\}$. 
Then, $S \in \gtpart{G}{T}$ holds, and $\coupgt{\hat{G}}{\hat{T}}{S}  = \coupgt{G}{T}{S}$. 
\end{lemma} 
\begin{proof} 
According to Theorem~\ref{thm:icomp}, 
there exist $S_1, \ldots, S_k \in \tpart{\hat{G}}{\hat{T}}$ 
such that $\agtr{\hat{G}}{\hat{T}}{r} = S_1 \dot\cup \cdots \dot\cup S_k$, where $k\ge 1$,  
and $\dgtr{\hat{G}}{\hat{T}}{r} = \coupgt{\hat{G}}{\hat{T}}{S_1} \dot\cup \cdots \dot\cup \coupgt{\hat{G}}{\hat{T}}{S_k}$. 
From Lemma~\ref{lem:extend}, 
there exists $i\in \{1, \ldots, k\}$ with $S_i = \{r\}$; without loss of generality, we assume $i = 1$. 
We also have $\coupgt{\hat{G}}{\hat{T}}{S_1} = \{ s\}$. 
Now, let $i\in \{1, \ldots, k\}\setminus \{1\}$.
Under Lemma~\ref{lem:extend2sim}, let $S_i'$ be the member of $\tpart{G}{T}$ with $S_i \subseteq S_i'$, 
and let $C$ be the member of $\tpart{\hat{G}}{\hat{T}}$ with $S_i' \subseteq  V(C)$. 
Lemma~\ref{lem:fcomp2ar} implies   $V(C) \subseteq S_i \cup \coupgt{\hat{G}}{\hat{T}}{S_i}$.  
Hence, we have  $ S_i'\setminus S_i  \subseteq \coupgt{\hat{G}}{\hat{T}}{S_i}$.  
However, because $\coupgt{\hat{G}}{\hat{T}}{S_i}$ is disjoint from $s$,  
we have $( S_i' \setminus S_i )\cap \coupgt{\hat{G}}{\hat{T}}{S_i} = \emptyset$. 
This implies $S_i' = S_i$, which further implies $\coupgt{\hat{G}}{\hat{T}}{S_i} = \coupgt{G}{T}{S_i}$.  
The lemma is proved. 
\end{proof}

Theorem~\ref{thm:icomp} and Lemma~\ref{lem:extend2icomp} easily imply Lemma~\ref{lem:hsetar2partition}.

\begin{lemma} \label{lem:hsetar2partition} 
Let $(G, T)$ be a bipartite graft with color classes $A$ and $B$, and let $X \subseteq A$ be an extreme set. 
Then, there exist $S_1, \ldots, S_k \in\tpart{G}{T}$, where $k\ge 1$, 
such that $\agtr{G}{T}{X} = S_1 \dot\cup \cdots \dot\cup S_k$. 
Furthermore, 
$\dgtr{G}{T}{X} = \coupgt{G}{T}{S_1} \dot\cup \cdots \dot\cup  \coupgt{G}{T}{S_k}$ and 
$\conng{G}{\dgtr{G}{T}{X}} = \conng{G}{\coupgt{G}{T}{S_1}} \dot\cup \cdots \dot\cup \conng{G}{\coupgt{G}{T}{S_k}}$. 
\end{lemma}

\section{Union of Initial Components}

We first provide and prove Lemma~\ref{lem:ad2include} and then use this lemma to derive Lemma~\ref{lem:ad2union}.

\begin{lemma} \label{lem:ad2include} 
Let $(G, T)$ be a bipartite graft, and let $r\in V(G)$. 
Let $x\in \agtr{G}{T}{r}$. 
Then, $\agtr{G}{T}{x} \subseteq \agtr{G}{T}{r}$ and $\dgtr{G}{T}{x} \subseteq \dgtr{G}{T}{r}$ hold. 
\end{lemma} 
\begin{proof} 
Let $F$ be a minimum join of $(G, T)$. 

\begin{pclaim} \label{claim:initial} 
$\agtr{G}{T}{x} \cup \dgtr{G}{T}{x} \subseteq \agtr{G}{T}{r}\cup \dgtr{G}{T}{r}$.
\end{pclaim} 
\begin{proof} 
Let $y\in  \parNei{G}{\initialgtr{G}{T}{r}}$.   
Theorem~\ref{thm:neigh2dist} implies $\distgtf{G}{T}{F}{r}{y} = 1$. 
Because Lemma~\ref{lem:ar} implies $\distgtf{G}{T}{F}{x}{y} \ge \distgtf{G}{T}{F}{r}{y}$, 
we obtain $\distgtf{G}{T}{F}{x}{y} \ge 1$. 
Consequently, $\distgtf{G}{T}{F}{x}{y} \ge 1$ holds for every $y\in  \parNei{G}{\initialgtr{G}{T}{r}}$.   
Therefore, $\initialgtr{G}{T}{x}$ is a subgraph of $\initialgtr{G}{T}{r}$, 
and the claim is proved. 
\end{proof} 

Under Claim~\ref{claim:initial}, we further prove the following claims. 
\begin{pclaim} 
$\dgtr{G}{T}{x} \subseteq \dgtr{G}{T}{r}$.   
\end{pclaim} 
\begin{proof} 
For every $y \in \dgtr{G}{T}{x}$,  Lemma~\ref{lem:ar} implies $\distgtf{G}{T}{F}{r}{y} \le \distgtf{G}{T}{F}{x}{y} \le  -1$. 
By Claim~\ref{claim:initial}, this implies $\dgtr{G}{T}{x} \subseteq \dgtr{G}{T}{r}$.  
\end{proof}

\begin{pclaim} 
$\agtr{G}{T}{x} \subseteq \agtr{G}{T}{r}$.   
\end{pclaim} 
\begin{proof} 
Next, let $y\in \agtr{G}{T}{x}$. 
By applying Lemma~\ref{lem:ar} to $\agtr{G}{T}{x}$, 
we obtain that $\distgtf{G}{T}{F}{y}{r} \ge \distgtf{G}{T}{F}{x}{r} = 0$. 
Thus, Claim~\ref{claim:initial} implies $y\in \agtr{G}{T}{r}$.   
Therefore,   $\agtr{G}{T}{x} \subseteq \agtr{G}{T}{r}$ is proved.   
\end{proof} 
This completes the proof.

\end{proof}

Lemma~\ref{lem:ad2include} easily implies Lemma~\ref{lem:ad2union}.

\begin{lemma} \label{lem:ad2union} 
Let $(G, T)$ be a bipartite graft, and let $r\in V(G)$.  
Then, 
$\agtr{G}{T}{r} = \bigcup_{x\in \agtr{G}{T}{r} } \agtr{G}{T}{x}$ 
and $\dgtr{G}{T}{r} = \bigcup_{x\in \dgtr{G}{T}{r} } \dgtr{G}{T}{x}$. 
\end{lemma}

Lemma~\ref{lem:dist2s} can easily be confirmed from Lemma~\ref{lem:extend}. 

\begin{lemma} \label{lem:dist2s} 
Let $(G, T)$ be a bipartite graft with color classes $A$ and $B$, and let $X \subseteq A$ be an extreme set.  
Let $(\hat{G}, \hat{T}):= \extend{G}{T}{X}{r}{s}$.   
Let $F\subseteq E(G)$ be a minimum join of $(G, T)$, and let $\hat{F} := F \cup \{rs\}$. 
Let $x\in V(G)$. 
Then, $\distgtf{\hat{G}}{\hat{T}}{\hat{F}}{x}{s} = \distgtf{\hat{G}}{\hat{T}}{\hat{F}}{x}{r} + 1$. 
\end{lemma}

Lemma~\ref{lem:dist2s} easily implies Lemma~\ref{lem:ar2disjoint}.

\begin{lemma} \label{lem:ar2disjoint} 
Let $(G, T)$ be a bipartite graft with color classes $A$ and $B$, and let $X \subseteq A$ be an extreme set.  
Let $(\hat{G}, \hat{T}):= \extend{G}{T}{X}{r}{s}$.   
Let $F\subseteq E(G)$ be a minimum join of $(G, T)$, and let $\hat{F} := F \cup \{rs\}$. 
 Let $x\in \agtr{\hat{G}}{\hat{T}}{r} \setminus \{r\}$. 
 Then, 
$\distgtf{\hat{G}}{\hat{T}}{\hat{F}}{x}{s} \ge 1$ holds. 
Accordingly, 
$( \agtr{\hat{G}}{\hat{T}}{x} \cup  \dgtr{\hat{G}}{\hat{T}}{x}) \cap \{s, r\} = \emptyset$,  
$\agtr{\hat{G}}{\hat{T}}{x} = \agtr{G}{T}{x}$, 
and $\dgtr{\hat{G}}{\hat{T}}{x} = \dgtr{G}{T}{x}$. 
\end{lemma}

Lemmas~\ref{lem:ad2include} and \ref{lem:ar2disjoint} easily imply Lemma~\ref{lem:rootad2include}.

\begin{lemma} \label{lem:rootad2include} 
Let $(G, T)$ be a bipartite graft with color classes $A$ and $B$, and let $X \subseteq A$ be an extreme set. 
Let $x\in \agtr{G}{T}{X}$.  
Then, $\agtr{G}{T}{x} \subseteq \agtr{G}{T}{X}$ and $\dgtr{G}{T}{x} \subseteq \dgtr{G}{T}{X}$ hold. 
\end{lemma} 
\begin{proof} 
Let $(\hat{G}, \hat{T}) := \extend{G}{T}{X}{r}{s}$. 
Note $\agtr{\hat{G}}{\hat{T}}{r} = \agtr{G}{T}{X} \dot\cup \{r\}$ and $\dgtr{\hat{G}}{\hat{T}}{r} = \dgtr{G}{T}{X} \dot\cup \{s\}$.  
Lemma~\ref{lem:ad2include} implies $\agtr{\hat{G}}{\hat{T}}{x} \subseteq \agtr{\hat{G}}{\hat{T}}{r}$  
and $\dgtr{\hat{G}}{\hat{T}}{x} \subseteq \dgtr{\hat{G}}{\hat{T}}{r}$. 
From Lemma~\ref{lem:ar2disjoint}, 
these imply 
$\agtr{\hat{G}}{\hat{T}}{x} \subseteq \agtr{G}{T}{X}$  
and $\dgtr{\hat{G}}{\hat{T}}{x} \subseteq \dgtr{G}{T}{X}$.  
Additionally, $\agtr{\hat{G}}{\hat{T}}{x} = \agtr{G}{T}{x}$, and $\dgtr{\hat{G}}{\hat{T}}{x} = \dgtr{G}{T}{x}$.
Accordingly, the claim is proved. 
\end{proof}

Lemma~\ref{lem:ar2disjoint} immediately implies Lemma~\ref{lem:rootad2union}.

\begin{lemma} \label{lem:rootad2union} 
Let $(G, T)$ be a bipartite graft with color classes $A$ and $B$, and let $X \subseteq A$ be an extreme set.  
Then, 
$\agtr{G}{T}{X} = \bigcup_{x\in \agtr{G}{T}{X} } \agtr{G}{T}{x}$ 
and $\dgtr{G}{T}{X} = \bigcup_{x\in \dgtr{G}{T}{X} } \dgtr{G}{T}{x}$  hold. 
\end{lemma}

\section{Initial Subgraph of Heterogeneous Extreme Root Set}

In this section, we provide and prove Lemmas~\ref{lem:icomp2disjoint}, \ref{lem:bi2ext}, and \ref{lem:heteroad2disjoint}  
and then derive Theorem~\ref{thm:hetero}.  
Lemma~\ref{lem:icomp2disjoint} is implied from Lemma~\ref{lem:ar}.

\begin{lemma} \label{lem:icomp2disjoint}
Let $(G, T)$ be a bipartite graft, and let $r\in V(G)$.   
Let $x\in V(G)$ be a vertex with $\distgt{G}{T}{r}{x} > 0$.  
Then, $\initialgtr{G}{T}{x}$ and  $\initialgtr{G}{T}{r}$ are disjoint. 
\end{lemma} 
\begin{proof}  
Lemma~\ref{lem:ar} implies $\distgt{G}{T}{r'}{x} > 0$ for every $r' \in \agtr{G}{T}{r}$. 
This implies the claim. 
\end{proof}

Lemma~\ref{lem:bi2ext} is implied from Lemmas~\ref{lem:ar} and \ref{lem:extend}.

\begin{lemma} \label{lem:bi2ext} 
Let $(G, T)$ be a bipartite graft with color classes $A$ and $B$, and let $X \subseteq V(G)$ be an extreme set. 
Then,  
$\agtr{G}{T}{X \cap A} \cup \agtr{G}{T}{X \cap B}$ is extreme. 
\end{lemma}
\begin{proof} 
Let $F$ be a minimum join of $(G, T)$.  

\begin{pclaim} \label{claim:bi2ext} 
The following two properties hold: 
\begin{rmenum} 
\item \label{item:bi2ext:a} For every $x \in \agtr{G}{T}{X\cap A}$ and every $z\in V(G)$, 
there exists $x_0 \in X\cap A$ with $\distgtf{G}{T}{F}{x}{z} \ge \distgtf{G}{T}{F}{x_0}{z}$.   
\item \label{item:bi2ext:b} For every $y \in \agtr{G}{T}{X\cap B}$ and every $z\in V(G)$,  
there exists $y_0 \in X\cap B$ with $\distgtf{G}{T}{F}{z}{y} \ge \distgtf{G}{T}{F}{z}{y_0}$.   
\end{rmenum} 
\end{pclaim} 
\begin{proof} 
First, we prove \ref{item:bi2ext:a}.  
Let $(\hat{G}, \hat{T}) := \extend{G}{T}{ X\cap A }{r}{s}$.  
Let $\hat{F} := F \cup \{ rs\}$. Lemma~\ref{lem:extend} implies that $F$ is a minimum join of $(\hat{G}, \hat{T})$.

Let $z\in V(G)$. 
From Lemmas~\ref{lem:ar} and \ref{lem:extend},  for every $x \in \agtr{\hat{G}}{\hat{T}}{r}$, 
 $\distgtf{\hat{G}}{\hat{T}}{\hat{F}}{x}{z} \ge \distgtf{\hat{G}}{\hat{T}}{\hat{F}}{r}{z} = \min_{x \in X\cap A} \distgtf{G}{T}{F}{x}{z}$.  

Lemma~\ref{lem:extend} also implies $\distgtf{\hat{G}}{\hat{T}}{\hat{F}}{x}{z} \le \distgtf{G}{T}{F}{x}{z}$.   
Thus, we have $\distgtf{G}{T}{F}{x}{z} \ge \min_{x' \in X\cap A} \distgtf{G}{T}{F}{x'}{z}$. 
This implies \ref{item:bi2ext:a}. 

The same argument proves \ref{item:bi2ext:b}. 
\end{proof}

Let $x \in \agtr{G}{T}{X\cap A}$ and $y \in \agtr{G}{T}{X\cap B}$. 
Furthermore, by Claim~\ref{claim:bi2ext} \ref{item:bi2ext:a}, there exists a vertex $x_0 \in X\cap A$ with 
$\distgtf{G}{T}{F}{x}{y} \ge \distgtf{G}{T}{F}{x_0}{y}$. 
From Claim~\ref{claim:bi2ext} \ref{item:bi2ext:b},  
there exists a vertex $y_0 \in X\cap B$ with  $\distgtf{G}{T}{F}{x}{y} \ge \distgtf{G}{T}{F}{x}{y_0}$. 
Because $X\cup Y$ is extreme, we also have $\distgtf{G}{T}{F}{x_0}{y_0} > 0$. 
These inequalities imply $\distgtf{G}{T}{F}{x}{y} >0$. 
Thus, it follows that $\agtr{G}{T}{X\cap A} \cup \agtr{G}{T}{X\cap B}$ is extreme.  
The lemma is proved. 
\end{proof}

Lemmas~\ref{lem:rootad2union}, \ref{lem:icomp2disjoint}, and \ref{lem:bi2ext} imply Lemma~\ref{lem:heteroad2disjoint}.

\begin{lemma} \label{lem:heteroad2disjoint} 
Let $(G, T)$ be a bipartite graft with color classes $A$ and $B$, and let $X \subseteq V(G)$ be an extreme set. 
Then, 
$\agtr{G}{T}{X} = \agtr{G}{T}{X\cap A} \dot\cup \agtr{G}{T}{X\cap B}$, 
and $\dgtr{G}{T}{X} = \dgtr{G}{T}{X\cap A} \dot\cup \dgtr{G}{T}{X\cap B}$.    
\end{lemma} 
\begin{proof} 
From Lemma~\ref{lem:bi2ext},  
for every $x\in \agtr{G}{T}{X\cap A}$ and every $y\in \agtr{G}{T}{X\cap B}$, 
we have $\distgt{G}{T}{x}{y} > 0$. 
Hence, by Lemma~\ref{lem:icomp2disjoint}, it follows that  $\agtr{G}{T}{x} \cap \agtr{G}{T}{y} = \emptyset$ 
and $\dgtr{G}{T}{x} \cap \agtr{G}{T}{y} = \emptyset$. 
Therefore, Lemma~\ref{lem:rootad2union}  
implies $\agtr{G}{T}{X} = \agtr{G}{T}{X\cap A} \dot\cup \agtr{G}{T}{X\cap B}$  
and $\dgtr{G}{T}{X} = \dgtr{G}{T}{X\cap A} \dot\cup \dgtr{G}{T}{X\cap B}$.    
The lemma is proved. 
\end{proof}

 Lemma~\ref{lem:hsetar2partition} and \ref{lem:heteroad2disjoint} easily imply Theorem~\ref{thm:hetero}. 

\begin{theorem} \label{thm:hetero} 
Let $(G, T)$ be a bipartite graft with color classes $A$ and $B$, and let $X \subseteq V(G)$ be an extreme set. 
Then,  there exist $S_1, \ldots, S_k \in\tpart{G}{T}$, where $k\ge 1$, 
such that $\agtr{G}{T}{X} = S_1 \dot\cup \cdots \dot\cup S_k$, 
$\dgtr{G}{T}{X} = \coupgt{G}{T}{S_1} \dot\cup \cdots \dot\cup  \coupgt{G}{T}{S_k}$, and 
$\conng{G}{\dgtr{G}{T}{X}} = \conng{G}{\coupgt{G}{T}{S_1}} \dot\cup \cdots \dot\cup \conng{G}{\coupgt{G}{T}{S_k}}$.   
\end{theorem}

\section{Structure of General Capital Components}

In this section, 
we investigate the structure of  capital components other than the initial component. 
We provide and prove Lemma~\ref{lem:neigh2posi} and then derive Theorems~\ref{thm:hinitial} and \ref{thm:capital}.

\begin{lemma} \label{lem:neigh2posi} 
Let $(G, T)$ be a bipartite graft, and let $r\in V(G)$.  
Let $i$ be an index with $i < \max_{x\in V(G)} \distgt{G}{T}{r}{x}$, 
and let $L$ be the member of $\laycompgtr{G}{T}{r}{i}$ with $r\in V(K)$. 
Then, $\distgt{G}{T}{x}{y} \ge 1$ holds for every $x \in V(L)\cap \levelgtr{G}{T}{r}{i}$ and every $y\in \parNei{G}{L}$. 
\end{lemma} 
\begin{proof} 
Let $x\in V(L) \cap \levelgtr{G}{T}{r}{i}$ and $y\in \parNei{G}{L}$. Let $z\in V(L)\cap \parNei{G}{y}$. 
Note that Theorem~\ref{thm:sebocut} implies that $yz$ is not allowed. 
Note also that Theorem~\ref{thm:neigh2dist} implies $z\in \levelgtr{G}{T}{r}{i}$. 
Let $F$ be a minimum join of $(G, T)$. 
Let $Q$ be an $F$-shortest path in $(G, T)$ between $x$ and $y$.  
Trace $Q$ from $y$, and let $w$ be the first vertex in $V(L)$. 
Theorem~\ref{thm:neigh2dist} implies  $w \in V(L)\cap \levelgtr{G}{T}{r}{i}$. 
Therefore, Lemma~\ref{lem:exext} implies $w_F(xQw) \ge 0$.    

First, consider the case $yz \in E(Q)$. 
This implies $w = z$ and $Q = xQw + zy$. 
Because  $wy\not\in F$ holds,  $w_F(Q) \ge 1$ follows.

Next, consider the case $yz \not\in E(Q)$. 
If $w = z$, then  Lemma~\ref{lem:nonallowed2dist} implies $w_F(wQy) \ge 1$. 
Otherwise, that is, if $w\neq z$, 
then $wQy + yz$ is a round ear path relative to $L$. 
Thus, Lemma~\ref{lem:ear2posi} implies $w_F(wQy + yz) \ge 2$. 
Accordingly, $w_F(wQy) \ge 1$. 
Therefore, in either case,  we obtain $w_F(Q) = w_F(xQw) + w_F(wQy) \ge 1$.    
This proves the lemma.  
\end{proof}

Lemma~\ref{lem:neigh2posi} implies Theorem~\ref{thm:hinitial}.

\begin{theorem} \label{thm:hinitial} 
Let $(G, T)$ be a bipartite graft, and let $r\in V(G)$.  
Let $i$ be an index with $i < \max_{x\in V(G)} \distgt{G}{T}{r}{x}$, 
and let $K$ be the member of $\laycompgtr{G}{T}{r}{i}$ with $r\in V(K)$.   
Let $L$ be the member of $\laycompgtr{G}{T}{r}{i+1}$ with $r\in V(K)$.    
Then, 
$V(L) \cap \levelgtr{G}{T}{r}{i+1} = \agtr{G}{T}{\parNei{G}{K}}$ and 
$V(L) \setminus V(K) \setminus \levelgtr{G}{T}{r}{i+1} = \dgtr{G}{T}{\parNei{G}{K}}$.  
\end{theorem} 
\begin{proof} 
Let $F$ be a minimum join of $(G, T)$. 
We first prove the following claim. 

\begin{pclaim} \label{claim:include} 
Let $x\in V(L)\setminus V(K)$. 
Then, $L$ has a path  between $x$ and a vertex in $\parNei{G}{K}$ with nonpositive $F$-weight whose vertices are contained in $V(L)\setminus V(K)$.  
This path can be of negative $F$-weight if and only if $\distgt{G}{T}{r}{x} \le i$. 

\end{pclaim} 
\begin{proof} 
Let $P$ be an $F$-shortest path in $(G, T)$ between $r$ and $x$. Note $w_F(P)  \le i+1$. 
Theorem~\ref{thm:path2cut} implies $V(P)\subseteq V(L)$ and $|\parcut{G}{K}\cap E(P) | = 1$. 
Let $e\in \parcut{G}{K}\cap E(P)$, and let $s \in \partial_G(e) \setminus V(K)$. 
Because $s \in \levelgtr{G}{T}{r}{i+1}$ holds from Theorem~\ref{thm:neigh2dist}, we have $w_F(rPs) \ge i+1$. 
Hence,  $w_F(sPx) \le 0$.  
Thus, $sPx$ is a desired path. 
If $x$ further satisfies $x\not\in \levelgtr{G}{T}{r}{i+1}$, then $w_F(P) < i+1$ holds, and $w_F(sPx) < 0$ follows.  
In contrast,  
if $x$ is contained in $\levelgtr{G}{T}{r}{i+1}$, 
then Lemma~\ref{lem:inext} implies $w_F(sPx) \ge 0$, and $w_F(sPx) = 0$ follows. 
This proves the claim. 
\end{proof}

Claim~\ref{claim:include} implies that 
$V(L) \cap \levelgtr{G}{T}{r}{i+1}$ is a subset of $\agtr{G}{T}{\parNei{G}{K}}$, and 
$V(L)\setminus V(K) \setminus \levelgtr{G}{T}{r}{i+1}$ is a subset of $\dgtr{G}{T}{\parNei{G}{K}}$.

By contrast, Lemma~\ref{lem:neigh2posi} implies 
$\agtr{G}{T}{\parNei{G}{K}} \cup \dgtr{G}{T}{\parNei{G}{K}} \subseteq V(L)$. 
Lemma~\ref{lem:neigh2posi} also implies that 
$K$ is disjoint from $\agtr{G}{T}{\parNei{G}{K}}$ and $\dgtr{G}{T}{\parNei{G}{K}}$. 
Hence, $\agtr{G}{T}{\parNei{G}{K}} \cup \dgtr{G}{T}{\parNei{G}{K}} \subseteq V(L) \setminus V(K)$ is implied.    

Therefore, the lemma is proved. 
\end{proof}

Lemma~\ref{lem:hsetar2partition} and Theorem~\ref{thm:capital} easily imply Theorem~\ref{thm:capital}.

\begin{theorem} \label{thm:capital} 
Let $(G, T)$ be a bipartite graft, and let $r\in V(G)$.  
Let $i$ be an index with $i > \min_{x\in V(G)} \distgt{G}{T}{r}{x}$, 
and let $K$ be the member of $\laycompgtr{G}{T}{r}{i-1}$ with $r\in V(K)$.   
Then, there exist $S_1, \ldots, S_k\in \gtpart{G}{T}$, where $k\ge 1$, such that 
$V(L) \cap \levelgtr{G}{T}{r}{i} = S_1 \dot\cup \cdots \dot\cup S_k$.    
Furthermore, 
$V(L) \setminus V(K) \setminus \levelgtr{G}{T}{r}{i} = \coupgt{G}{T}{S_1} \dot\cup \cdots \dot\cup  \coupgt{G}{T}{S_k}$ and 
$\conn{L - V(K) -  \levelgtr{G}{T}{r}{i}} = \conng{G}{\coupgt{G}{T}{S_1}} \dot\cup \cdots \dot\cup \conng{G}{\coupgt{G}{T}{S_k}}$.  
\end{theorem}

\begin{ac} 
This study was supported by JSPS KAKENHI Grant Number 18K13451. 
\end{ac}

\bibliographystyle{splncs03.bst}
\bibliography{tbarrier.bib}

\end{document}